\documentclass[11pt]{article}
\usepackage[english]{babel}
\usepackage{amssymb}
\usepackage{amsmath}
\usepackage[top=1.2in, bottom=1.4in, left=1in, right=1in]{geometry}
\usepackage{parskip}
 \usepackage[T1]{fontenc}
\usepackage{lmodern}
 \usepackage{centernot}
\usepackage{epsfig}
 \usepackage{float}
 \usepackage{amsmath}
\usepackage{amsmath,amsthm,amsfonts,amssymb,amscd}
\usepackage{indentfirst}
\usepackage{dirtytalk}

\usepackage{mathrsfs}
\usepackage{amsthm}
\usepackage{tikz}
\usepackage{mathtools}
\usepackage{calc}
\usepackage[symbol]{footmisc}
\usepackage{todonotes}
\usetikzlibrary{patterns,arrows,decorations.pathreplacing}
\usetikzlibrary{fadings}
\usepackage{pgfplots}
\usepackage{parskip}
\theoremstyle{plain}
\newtheorem{theorem}{Theorem}

\newtheorem{claim}[theorem]{Claim}
\newtheorem{conjecture}{Conjecture}

\newcommand{\floor}[1]{\left\lfloor{#1}\right\rfloor}
\newcommand{\ceil}[1]{\left\lceil{#1}\right\rceil}

\usepackage{authblk}
\title{Book free $3$-Uniform Hypergraphs}
\author[1]{Debarun Ghosh}
\author[1,2]{Ervin Gy\H{o}ri} 
\author[5]{Judit Nagy-Gy\"{o}rgy}
\author[1]{Addisu Paulos}
\author[1]{Chuanqi Xiao}
\author[1,4]{Oscar Zamora}

\affil[1]{Central European University, Budapest\par
\texttt{ chuanqixm@gmail.com, ghosh\textunderscore debarun@phd.ceu.edu,addisu\textunderscore 2004@yahoo.com,oscarz93@yahoo.es}}
\affil[2]{Alfr\'ed R\'enyi Institute of Mathematics, Budapest \par
\texttt{gyori.ervin@renyi.mta.hu}}
\affil[4]{Universidad de Costa Rica, San Jos\'e}
\affil[5]{University of Szeged, Szeged\par
\texttt{nagy-gyorgy@math.u-szeged.hu}}
\date{}
\begin{document}

\maketitle
\begin{abstract}
A $k$-book in a hypergraph consists of $k$ Berge triangles sharing a common edge.    In this paper we prove that the number of the hyperedges in a $k$-book-free 3-uniform hypergraph on $n$ vertices is at most $\frac{n^2}{8}(1+o(1))$.
\end{abstract}

\section{ Introduction}
Let $G$ be a graph.  The vertex and the edge set of $G$ are denoted by $V(G)$ and $E(G)$. If there are two triangles sitting on an edge in a graph, we call this a \emph{diamond}. Whereas $k$ triangles sitting on an edge is called a \emph{$k$-book}, denoted by a $B_k$. Similarly, let $H$ be a hypergraph and the vertex and the edge set of $H$ be denoted by $V(H)$ and $E(H)$. A hypergraph is called \emph{$r$-uniform} if each member of $E$ has size $r$. A hypergraph $H = (V,E)$ is called \emph{linear} if every two hyperedges have at most one vertex in common.  A Berge cycle of length $k$, denoted by Berge-$C_k$, is an alternating sequence of distinct vertices and distinct hyperedges of the form $v_1, h_1, v_2, h_2, \dots, v_k, h_k$ where $v_i, v_{i+1}\in h_i$ for each $i \in \{1, 2, \dots , k-1\}$ and $v_kv_1 \in h_k$.  The hypergraph equivalent of $k$-books is defined similarly with $k$ Berge triangles sharing a common edge. We say that this common edge is the base of the $k$-book.

The maximum number of edges in a triangle-free graph is one of the classical results in extremal graph theory and proved by Mantel in $1907$ \cite{MAN}.  The extremal problem for diamond-free graphs follows from this.  Given a graph $G$ on $n$ vertices and having $\floor{\frac{n^2}{4}}+1$ edges.  Mantel showed that $G$ contains a triangle.  Rademacher (unpublished, and simplified later by Erd\H{o}s in \cite{erdos1955}) proved in the 1940s that the number of triangles in $G$ is at least $\floor{\frac{n}{2}}$.  Erd\H{o}s conjectured in 1962 \cite{erdos} that the size of the largest book in $G$ is at least $\frac{n}{6}$
and this was proved soon after by Edwards (unpublished, see also Khad\'ziivanov and Nikiforov \cite{khad}
for an independent proof). 
\begin{theorem}[Edwards \cite{edwards}, Khad\'ziivanov and Nikiforov \cite{khad}]\label{edwards}
Every $n$-vertex graph with more than $ \frac{n^2}{4}$ edges contains an edge that is in at least $\frac{n}{6}$ triangles.
\end{theorem}
Both Rademacher’s and Edwards’ results are sharp. In the former, the addition of an edge to one part in the complete balanced bipartite graph (note that in $G$ there is an edge contained in $\floor{\frac{n}{2}}$ triangles) achieves the maximum. In the latter, every known extremal construction of $G$  has $\Omega(n^3)$ triangles.  For more details on book-free graphs we refer the reader to the following articles  \cite{erdos1975}, \cite{qiao2021} and \cite{yan2021tur}. We look into the equivalent problem in the case of hypergraphs.

Given a family of hypergraphs $\mathcal{F}$, we say that a hypergraph $H$ is $\mathcal{F}$-free if for every $F \in \mathcal{F}$, the hypergraph $H$ does not contain a $F$ as a sub-hypergraph. 

The systematic study of the Tur\'an number of Berge cycles started with Lazebnik and Verstra\"ete \cite{lazebnik}, who studied the maximum number of hyperedges in an $r$-uniform hypergraph containing no Berge cycle of length less than five.  Another result was the study of Berge triangles by Gy\H{o}ri \cite{gyori1}.  He proved that:
\begin{theorem}[Gy\H{o}ri \cite{gyori1}]
The maximum number of hyperedges in a Berge triangle-free $3$-uniform hypergraphs on $n$ vertices is at most $\frac{n^2}{8}$.
\end{theorem}
It continued with the study of Berge five cycles by Bollob\'{a}s and Gy\H{o}ri \cite{gyori2}.  In \cite{gyorikatona}, Gy\H{o}ri, Katona, and Lemons proved the following analog of the Erd\H{o}s-Gallai Theorem \cite{erdosgallai} for Berge paths.  For other results see \cite{other1,other2}.  The particular case of determining the maximum number of the hyperedges of a triangle-free linear hypergraph on $n$ vertices 
is equivalent to the famous $(6,3)$-problem, which is a special case of a general problem of Brown, Erd\H{o}s, and S\'os. The following theorem of Ruzsa and Szemer\'edi plays important role in our paper:
\begin{theorem}[Ruzsa and Szemer\'edi \cite{trianglefreelinear}]\label{ruzsa}
For any $\epsilon>0$ the exists $n_0(\epsilon)$ such that if  $n>n_0(\epsilon)$ then a Berge-triangle-free 3-uniform linear hypergraph on $n$ vertices has at most $\epsilon n^2$ hyperedges.
\end{theorem}

We continue the work on that and determine the maximum number of hyperedges for a $k$-book-free $3$-uniform hypergraph.  The main result is as follows:

\begin{theorem}
\label{3-unif}
For a given $k\geq 2$ and $\epsilon>0$ there exists $n_1(k,\epsilon)$ such that if $n>n_1(k,\epsilon)$ then
a $3$-uniform $B_k$-free hypergraph $H$ on $n$ vertices can have at most  $\dfrac{n^2}{8}+\epsilon n^2$ edges. 
\end{theorem}
The following example shows that this result is asymptotically sharp.  Take a complete bipartite graph with color classes of size $\ceil{\frac{n}{4}}$ and $\floor{\frac{n}{4}}$ respectively.  Denote the vertices in each class with $x_i$ and $y_i$ respectively. Construct a graph by doubling each vertex and replacing each edge with two hyperedges as shown below (Figure \ref{extremal_hyp}). So essentially, we have replaced every graph edge with two hyperedges. The construction does not contain a $B_k$, as it does not contain a Berge triangle. With this, the number of hyperedges is $2\times \frac{n^2}{16}=\frac{n^2}{8}$.

\begin{figure}[ht]
\centering
\begin{tikzpicture}[scale=0.25]
\draw[fill=black](0,0)circle(8pt);
\draw[fill=black](0,6)circle(8pt);
\draw[rotate around={0:(0,0)},red] (0,0) ellipse (5 and 2);
\draw[rotate around={0:(0,6)},red] (0,6) ellipse (5 and 2);
\node at (0,-1){$x_i$};
\node at (0,7){$y_j$};
\draw[thick,black] (0,0)--(0, 6);
\end{tikzpicture}\qquad\qquad
\begin{tikzpicture}[scale=0.25]
\draw[fill=black](0,0)circle(8pt);
\draw[fill=black](0,6)circle(8pt);
\draw[fill=black](6,0)circle(8pt);
\draw[fill=black](6,6)circle(8pt);
\draw[rotate around={0:(3,0)},red] (3,0) ellipse (9 and 2);
\draw[rotate around={0:(3,6)},red] (3,6) ellipse (9 and 2);
\node at (-1,-1){$x_i$};
\node at (-1,7){$y_j$};
\node at (7,-1){$x'_i$};
\node at (7,7){$y'_j$};

\draw[thick,black] (6,6)--(0,0)--(0, 6)--(6,0);
\draw[thick,black] (0,0)--(0, 6);
\draw[thick,black] (0,6)--(6, 6);
\draw[thick,black] (6,0)--(6, 6);

\end{tikzpicture}
\caption{Replacing every edge $x_iy_i$ in the bipartite graph with two hyperedges $x_iy_jy'_j$ and $y_jy'_jx'_i$}
\label{extremal_hyp}
\end{figure}
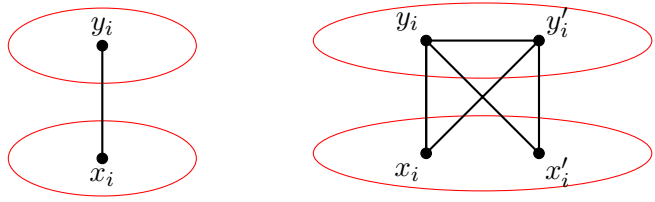

\section{Proof of Theorem \ref{3-unif}}

Fix $k\ge 2$, $\epsilon>0$ and set $$n_1(k,\epsilon)=\max\left\{18k+12, n_0\left(\frac{\epsilon}{6k^2-8k}\right)\right\}$$ where $n_0(.)$ is from Theorem~\ref{ruzsa}. Suppose that $n>n_1(k,\epsilon)$.
Let $H$ be a $B_k$-free $3$-uniform hypergraph on $n$ vertices.  We are interested in the $2$-shadow, i.e., let $G$ be a graph with vertex set $V(H)$ and $$E(G) = \{ab \mid \{a,b\} \subset e \in E(H)\}.$$  If an edge in $G$ lies in more than one hyperedge in $H$, we color it blue.  Otherwise, we color it red.  
We define hypergraphs $H_r$ and $H_b$ in the following way. $V(H_r)=V(H_b)=V(H)$, $$E(H_r)=\{e\in E(H) \mid e \textrm{ contains two or three red edges of } G\}$$ and $E(H_b)=E(H)\setminus E(H_r)$. Note that each hyperedge in $H_b$ contains two or three blue edges of $G$.

\begin{claim}\label{H_1}
The number of hyperedges in $H_r$ is at most $\frac{n^2}{8}$.
\end{claim}
\begin{proof}
Denote the subgraph of $G$ formed by the red colored edges by $G_r$.  Suppose that $|E(G_r)|\geq \frac{n^2}{4}+1$.  By Theorem~\ref{edwards} we have a book of size $\frac{n}{6}$ in $G_r$.  Denote the vertices of the $\frac{n}{6}$-book in $G_r$ with $u,v$ and $x_i$, $1\leq i\leq \frac{n}{6}$ respectively where $uv$ is the base of the book. Denote the third vertex of the hyperedge containing edge $uv$ by $w$, set $X=\{x_i\mid 1\leq i\leq \frac{n}{6}\}$ and for each $x_i\in X$ denote the hyperedge containing $ux_i$  and $vx_i$ by $ux_iy_i$ and $vx_iz_i$ respectively. 

Set $E':=\emptyset$ and $X':=\emptyset$.
Go through the vertices of $X$ and perform the following procedure for each of them. At the beginning of the process no vertex is marked.

If the current vertex $x_i=w$ then mark it.

If $x_i$ is unmarked then
\begin{itemize}
    \item add $x_i$ to $X'$ and hyperedges $ux_iy_i$ and $vx_iz_i$ to $E'$,  
    \item if there exists $j>i$ such that $y_i=x_j$ then mark $x_j$,
    \item if there exists $\ell>i$ such that $z_i=x_\ell$ then mark $x_\ell$.
\end{itemize}
By definition of red edges and  the procedure (i.e. it adds two new hyperedges to  $E'$ forming a Berge triangle with $uvw$ at each step handling an unmarked vertex but at most one: when $x_i=w$) the set of hyperedges $E'\cup\{uvw\}$ with vertex set $X'\cup\{u,v\}$ form a $k'$-book with base $uvw$, where $k'=|X'|$.
Moreover at each step of the procedure whenever an unmarked vertex was added to $X'$ then at most two more vertices became marked. Each unmarked vertex are in $X'$ at the end of the procedure, therefore $$k'\ge \frac{|X\setminus\{w\}|}{3}\ge\frac{n/6-1}{3}$$ at the end of the procedure and it is at least $k$ by the definition of $n_1(k,\epsilon)$, but this is a contradiction.

Hence $|E(G_r)|\leq \frac{n^2}{4}$ and $$|E(H_r)|\le\frac{|E(G_r)|}{2}\le\frac{n^2}{8}$$ by the definition of red colored edges.
\end{proof}

Now let us work on the sub-hypergraph $H_b$.
\begin{claim}\label{H_2_hyperedge}
Each pair of vertices is contained in at most $2k-2$ hyperedges of $H_b$.
\end{claim}
\begin{proof}
Suppose that $\{u,v\}$ is a pair of vertices which is contained in $2k-1$ hyperedges of $H_b$.  Note that edge $uv$ is colored blue.  Denote the third vertices of hyperedges containing $u$ and $v$ by $x_1,\ldots, x_{2k-1}$ and set $X=\{x_i\mid 1\leq i\leq 2k-1\}$. Observe that for each $i$ at least one of $ux_i$ and $vx_i$ is colored blue. 

Set $E':=\emptyset$ and $X':=\emptyset$. Go through the vertices of $X$ and perform the following procedure for each of them. At the beginning of the process no vertex is marked.

If the current vertex $x_i=x_{2k-1}$ and there is no marked vertex in $X$ then do nothing.

Otherwise if $x_i$ is unmarked then 
\begin{itemize}
    \item add $x_i$ to $X'$ and add $ux_iv$ to $E'$,
    \item if $ux_i$ is colored blue denote a hyperedge containing it by $ux_iy_i$ where $y_i\ne v$ and add $ux_iy_i$ to $E'$,  
    \item otherwise $vx_i$ is colored blue, so denote a hyperedge containing it by $vx_iy_i$ where $y_i\ne u$ and add $vx_iy_i$ to $E'$,  
    \item if there exists $j>i$ such that $y_i=x_j$ then mark $x_j$.
\end{itemize}

If at the end of the procedure there is no marked vertex in $X$ then set $w=x_{2k-1}$ otherwise let $w$ be an arbitrary marked vertex.

By definition of blue edges and  the procedure (i.e. it adds two new hyperedges to  $E'$ forming a Berge triangle with $uvw$ at each step handling an unmarked vertex but at most the last one) the set of hyperedges $E'\cup\{uvw\}$ with vertex set $X'\cup\{u,v\}$ form a $k'$-book with base $uvw$ where $k'=|X'|$. 
Moreover if there is no marked vertex in $X$ at the end of the process then $X'=X\setminus \{x_{2k-1}\}$, otherwise at each step of the procedure whenever an unmarked vertex was added to $X'$ than at most one more vertex became marked and each unmarked vertex are in $X'$ at the end of the procedure. Therefore $k'\ge k$, but it is a contradiction.
\end{proof}

We now give an upper bound on the number of hyperedges in $H_b$.
\begin{claim}\label{H_2}
The number of hyperedges in $H_b$ is at most $\epsilon n^2$.
\end{claim}

\begin{proof}
Take a hyperedge $xyz$ in the sub-hypergraph $H_b$.  By Claim~\ref{H_2_hyperedge} there are at most $2k-2$ hyperedges of $H_b$ containing each of the pairs of vertices $xy$, $yz$, and $xz$. 
If we deleted all such hyperedges barring $xyz$ we would delete at most $6k-9$ hyperedges. Therefore there is a linear 3-uniform subhypergraph $H'_b$ of $H_b$ with $V(H_b')=V(H_b)=V(H)$ and $$|E(H_b')|\ge \frac{|E(H_b)|}{6k-8}$$ (i.e. a greedy algorithm can find an appropriate $H_b'$).

Consider a hyperedge $e$ in $H_b'$.  Observe that $H_b'$ is a $B_k$-free hypergraph, since it is a subhypergraph of $H$,
therefore the number of Berge triangles sitting on the edge $e$ is at most $k-1$. 
Apply the following greedy procedure until all the hyperedges are marked. In a step pick an unmarked hyperedge, mark it and delete an unmarked hyperedge of each Berge triangle containing the current hyperedge. Observe that this marked edge is not an edge of a triangle anymore.
Define $H_b''$ the following way. Let $V(H_b'')=V(H_b')=V(H)$ and $E(H_b'')$ contains the remaining hyperedges of $H_b'$. Observe that at most $k-1$ edges were deleted in each step and marked edges were never deleted. Therefore $$|E(H_b'')|\ge \frac{|E(H_b')|}{k}.$$ 
Moreover $H_b''$ is a Berge-triangle-free 3-uniform linear hypergraph therefore Theorem~\ref{ruzsa} can be applied with $\epsilon'=\frac{\epsilon}{6k^2-8k}$. We get that
$$\frac{|E(H_b)|}{6k^2-8k}\le |E(H_b'')| \le \frac{\epsilon n^2}{6k^2-8k}.$$
\end{proof}

\begin{proof}[Proof of Theorem \ref{3-unif}]
By definition $E(H)$ is a disjoint union of $E(H_r)$ and $E(H_b)$. By Claim~\ref{H_1} and Claim~\ref{H_2}, $$|E(H)|\leq |E(H_r)|+|E(H_b)|\leq \frac{n^2}{8}+\epsilon n^2.$$ 
\end{proof}
\section{Conclusions}
Recall that both Tur\'an numbers of triangle-free graph and $k$-book-free graphs on $n$ vertices are $\frac{n^2}{4}$, moreover
Gy\H{o}ri \cite{gyori1} proved that the maximum number of hyperedges in a Berge triangle-free $3$-uniform hypergraphs on $n$ vertices is at most $\frac{n^2}{8}$.  Given the similarities, we conjecture the following:
\begin{conjecture}
For a given $k\geq 2$ every $3$-uniform $B_k$-free hypergraph $H$ on $n$ vertices ($n$ is large) has at most $\dfrac{n^2}{8}$ hyperedges.
\end{conjecture}
\section{Acknowledgements}
Gy\H{o}ri's research was partially supported by the National Research, Development and Innovation Office NKFIH, grants  K132696, K116769, and K126853. Judit Nagy-Gy\"orgy acknowledges support by the project TKP2021-NVA-09. Project no. TKP2021-NVA-09 has been implemented with the support provided by the Ministry of Innovation and Technology of Hungary from the National Research, Development and Innovation Fund, financed under the TKP2021-NVA funding scheme. Nagy-Gy\"orgy's	research was partially supported by the National Research, Development and Innovation Office NKFIH, grant 
KH129597.
 
\bibliographystyle{abbrv}
\bibliography{references}
\end{document}